\documentclass[letterpaper, 10pt, conference]{ieeeconf}

\IEEEoverridecommandlockouts        
\usepackage{hyperref}

\title{\LARGE \bf Synchronization Patterns in Networks of Kuramoto
  Oscillators: \\A Geometric Approach for Analysis and Control}

\author{Lorenzo Tiberi, Chiara Favaretto, Mario Innocenti, Danielle S.
  Bassett, and Fabio Pasqualetti%
  \thanks{This material is based upon work supported by NSF award
    \#BCS-1430279 and ARO award 71603NSYIP. Lorenzo Tiberi and Fabio
    Pasqualetti are with the Mechanical Engineering Department,
    University of California at Riverside, \href{mailto:
      ltiberi@ucr.edu}{\texttt{ltiberi@ucr.edu}},
    \href{mailto:fabiopas@engr.ucr.edu}{\texttt{fabiopas@engr.ucr.edu}}. Chiara
    Favaretto is with the Department of Information Engineering,
    University of Padova, \href{mailto:
      chiara.favaretto.2@phd.unipd.it}{\texttt{chiara.favaretto.2@phd.unipd.it}}. Danielle
    S. Bassett is with the Departments of Bioengineering and
    Electrical and Systems Engineering, University of Pennsylvania,
    \href{mailto:dsb@seas.upenn.edu}{\texttt{dsb@seas.upenn.edu}}.
    Mario Innocenti is with the Department of Electrical Systems and
    Automation, University of Pisa,
    \href{mailto:mario.innocenti@unipi.it}{\texttt{mario.innocenti@unipi.it}}.}}

\usepackage{graphicx,color}
\usepackage{amsmath}
\usepackage{amssymb}
\usepackage{mathrsfs}
\usepackage[vlined,ruled]{algorithm2e}
\usepackage{subfigure}
\usepackage{url}
\usepackage{color}
\usepackage{algorithmic}
\usepackage{dsfont}
\usepackage{bbm}

\usepackage{booktabs}
\usepackage{array}
\usepackage[table]{xcolor}
\usepackage{yfonts}
\usepackage{tikz}
\usetikzlibrary{shapes,snakes,3d,calc,arrows}
\usetikzlibrary{positioning,calc}
\usetikzlibrary{decorations,decorations.markings,decorations.text}

\newtheorem{theorem}{Theorem}[section]

\newtheorem{definition}{Definition}
\newtheorem{corollary}[theorem]{Corollary}

\newtheorem{remark}{Remark}

\newtheorem{example}{Example}

\newcommand{\setdef}[2]{\{#1 \; : \; #2\}}

\newcommand{\until}[1]{\{1,\dots,#1\}}

\newcommand{\Image}{\operatorname{Im}}

\newcommand{\real}{\mathbb{R}}

\newcommand{\transpose}{\mathsf{T}} 

\newcommand{\mc}{\mathcal}

\newcommand{\A}{\bold{A}}

\DeclareSymbolFont{bbold}{U}{bbold}{m}{n}
\DeclareSymbolFontAlphabet{\mathbbold}{bbold}

\newcommand\oprocendsymbol{\hbox{$\square$}}
\newcommand\oprocend{\relax\ifmmode\else\unskip\hfill\fi\oprocendsymbol}

\usepackage{tikz}
\usetikzlibrary{matrix}
\usepackage{amsmath}

\renewcommand{\theenumi}{(\roman{enumi})}

\begin{document}
\maketitle

\thispagestyle{empty}
\pagestyle{empty}

\begin{abstract}
  Synchronization is crucial for the correct functionality of many
  natural and man-made complex systems. In this work we characterize
  the formation of synchronization patterns in networks of Kuramoto
  oscillators. Specifically, we reveal conditions on the network
  weights and structure and on the oscillators' natural frequencies
  that allow the phases of a group of oscillators to evolve
  cohesively, yet independently from the phases of oscillators in
  different clusters. Our conditions are applicable to general
  directed and weighted networks of heterogeneous
  oscillators. Surprisingly, although the oscillators exhibit
  nonlinear dynamics, our approach relies entirely on tools from
  linear algebra and graph theory. Further, we develop a control
  mechanism to determine the smallest (as measured by the Frobenius
  norm) network perturbation to ensure the formation of a desired
  synchronization pattern. Our procedure allows us to constrain the
  set of edges that can be modified, thus enforcing the sparsity
  structure of the network perturbation. The results are validated
  through a set of numerical examples.
\end{abstract}

\section{Introduction}
Synchronization of coupled oscillators is everywhere in nature
\cite{Lewis2014, Strogatz2000, Ferrari2015} and in several man-made
systems, including power grids \cite{Doerfler2014} and computer
networks \cite{Nishikawa2015}. While some systems require complete
synchronization among all the parts to function properly
\cite{Danino2010, Kim2010}, others rely on cluster or partial
synchronization \cite{DAP-NEL-RS-DG-JKP:07}, where subsets of nodes
exhibit coherent behaviors that remain independent from the evolution
of other oscillators in the network. For example, while partial
synchronization patterns have been observed in healthy individuals
\cite{Schnitzler2005}, complete synchronization in neural systems is
associated with degenerative diseases including Parkinson's and
Huntington's diseases \cite{Hammond2007,Banaie2009}, and epilepsy
\cite{Lehnertz2009}. Cluster synchronization has received attention
only recently, and several fundamental questions remain unanswered,
including the characterization of the network features enabling the
formation of a desired pattern, and the development of control
mechanisms to enforce the emergence of clusters.

In this paper we focus on networks of Kuramoto oscillators
\cite{Kuramoto1975}, and we characterize intrinsic and topological
conditions that ensure the formation of desired clusters of
oscillators. Our network model is motivated by a large body of
literature showing broad applicability of the Kuramoto model to
virtually all systems that exhibit synchronization
properties. Although Kuramoto networks exhibit nonlinear dynamics, we
adopt tools from linear algebra and graph theory to characterize
network conditions enabling the formation of a given synchronization
pattern. Further, we design a control mechanism to perturb (a subset
of) the network weights so as to enforce or prevent desired
synchronization patterns.

\noindent
\textbf{Related work} Complete synchronization in networks of Kuramoto
oscillators has been extensively studied, e.g., see
\cite{Gardenes2007,Zhang2014}. It has been shown that synchronization of all nodes
emerges when the coupling strength among the agents is sufficiently
larger than the heterogeneity of the oscillators' natural
frequencies. Partial synchronization and pattern formation have
received considerably less attention, with the literature being
composed of only few recent works. In \cite{Pecora2014} it is shown
how symmetry of the interconnections may lead to partial
synchronization. Methods based on graph symmetry have also been used
to find all possible clusters in networks of Laplacian-coupled
oscillators \cite{Sorrentino2016}. The relationship between
clusterization and network topology has been studied in \cite{Lu2010}
for unweighted interconnections. In \cite{Dahms2012}, the emergence
and the stability of groups of synchronized agents within a network
has been studied for different classes of dynamics, like delay-coupled
laser models and neuronal spiking models. Here, the approach of master
stability function has been used to characterize the results. In
\cite{CF-DSB-AC-FP:16,CF-AC-FP:17} the idea is put forth to study an
approximate notion of cluster synchronization in Kuramoto networks via
tools from linear systems theory. It is quantitatively shown how
cluster synchronization depends on strong intra-cluster and weak
inter-cluster connections, similarity with respect to the natural
frequencies of the oscillators within each cluster, and heterogeneity
of the natural frequencies of coupled oscillators belonging to
different subnetworks. With respect to this work, we focus on an exact
notion of cluster synchronization, identify necessary and sufficient
conditions on the network weights and oscillators' natural frequencies
for the emergence of a desired synchronization pattern, and exploit
our analysis to design a structural control algorithm for the
formation of a desired synchronization pattern.

The work that is closer to this paper is \cite{Schaub2016}, where the
authors relate cluster synchronization to the notion of an external
equitable partition in a graph. In fact, the notion of an external
equitable partition can be interpreted in terms of invariant subspaces
of the network adjacency matrix, a notion that we exploit in our
development. However, the analysis in \cite{Schaub2016} is carried out
with unweighted and undirected networks and, as we show in this paper,
the conditions in \cite{Schaub2016} may not be necessary when dealing
with directed and weighted networks. Further, our approach relies on
simple notions from linear algebra, and leads to the development of
our control algorithm for the formation of desired patterns.

\noindent
\textbf{Paper contributions} The contributions of this paper are
twofold. First, we consider a notion of exact cluster synchronization,
where the phases of the oscillators within each cluster remain equal to
each other over time, and different from the phases of the oscillators in
different clusters. We derive necessary and sufficient conditions for the
formation of a given synchronization pattern in directed and weighted
networks of Kuramoto oscillators. In particular we show
that cluster synchronization is possible if and only if (i) the
natural frequencies are equal within each cluster, and (ii) for each
cluster, the sum of the weights of the edges from every separate
group is the same for all nodes in the cluster. Second, we leverage
our characterization of cluster synchronization to develop a control
mechanism that modifies the network weights so as to ensure the
formation of a desired synchronization pattern. Our control method is
optimal, in the sense that it determines the smallest (measured by the
Frobenius norm) network perturbation for a given synchronization
pattern, and it guarantees the modification of only a desired subset
of the edge weights.

\noindent
\textbf{Paper organization} The rest of this paper is organized as
follows. Section \ref{sec:setup and preliminary} contains the problem
setup and some preliminary definitions. Section \ref{sec:
  synchronization conditions} contains our characterization of cluster
synchronization, and Section \ref{sec: control of cluster
  synchronization} contains our structural control algorithm for the
formation of a desired synchronization pattern.  Section \ref{future
  research and conclusions} concludes the paper.

\section{Problem setup and preliminary notions}
\label{sec:setup and preliminary}
Consider a network of heterogenous Kuramoto oscillators described by
the digraph $\mc G = (\mc V, \mc E)$, where $\mc V = \until{n}$
denotes the set of oscillators and
$\mc E \subseteq \mc V \times \mc V$ their interconnections. Let
$A = [a_{ij}]$ be the weighted adjacency matrix of $\mc G$, where
$a_{ij} \in \real$ if $(i,j) \in \mc E$ and $a_{ij} = 0$
otherwise. We assume that $\mc G$ is strongly connected \cite{CDG-GFR:01}.
Let $\theta_i \in \mathbb{R}$ denote the phase of the $i$-th
oscillator, whose dynamics reads as
\begin{align*}
  \dot \theta_i = \omega_i + \sum_{j =1}^n a_{ij} \sin (\theta_j -
  \theta_i) ,
\end{align*}
where $\omega_i$ is the natural frequency of the $i$-th
oscillator. The dynamics is a generalized version of the classic
Kuramoto model \cite{YK:75}.  Depending on the interconnection graph
$\mc G$, the adjacency matrix $A$, and the oscillators natural
frequencies, different oscillatory patterns are possible corresponding
to (partially) synchronized or chaotic states
\cite{mirchev2014cooperative}. In this work we are particularly
interested in the case where the phases of groups of oscillators
evolve cohesively within each group, yet independently from the phases
of oscillators in different groups. To formalize this discussion, let
$\mc P = \{\mc P_1, \dots , \mc P_m\}$ be a partition of $\mc V$, that
is, $\mc V = \cup_{i = 1}^m \mc P_i$ and
$\mc P_i \cap \mc P_j = \emptyset$ for all $i,j \in \until{m}$ with
$i \neq j$. We restrict our attention to the case $m > 1$. Throughout
the paper we will assume without loss of generality that, given
$\mc P = \{\mc P_1, \dots , \mc P_m\}$, the oscillators are labeled so
that
$\mc P_i = \{\sum_{j =1}^{i-1}|\mc P_j|+1 , \dots, \sum_{j =
  1}^{i}|\mc P_j| \}$, where $|\mc P_j|$ denotes the cardinality of
the set $\mc P_j$. While different notions of synchronization exist,
we will use the following definitions.

\begin{definition}{\bf \emph{(Phase synchronization)}}\label{def:
    phase sync}
  For the network of oscillators $\mc G = (\mc V, \mc E)$, the
  partition $\mc P = \{\mc P_1, \dots , \mc P_m\}$ is \emph{phase
    synchronizable} if, for some initial phases
  $\theta_1 (0) , \dots, \theta_n (0)$, it holds
  \begin{align*}
    \theta_i (t) = \theta_j (t),
  \end{align*}
  for all times $t \in \real_{\ge 0}$ and $i,j \in \mc P_k$, with
  $k \in \until{m}$. \oprocend
\end{definition}

\begin{definition}{\bf \emph{(Frequency synchronization)}}\label{def:
    frequency sync}
  For the network of oscillators $\mc G = (\mc V, \mc E)$, the
  partition $\mc P = \{\mc P_1, \dots , \mc P_m\}$ is \emph{frequency
    synchronizable} if, for some initial phases
  $\theta_1 (0) , \dots, \theta_n (0)$, it holds
  \begin{align*}
    \dot \theta_i (t) = \dot \theta_j (t),
  \end{align*}
  for all times $t \in \real_{\ge 0}$ and $i,j \in \mc P_k$, with
  $k \in \until{m}$. \oprocend
\end{definition}
Clearly, phase synchronization implies frequency synchronization,
while the converse statement typically fails to hold.



Finally, we define the characteristic matrix associated with a
partition $\mc P$ of the network nodes, which will be used to derive
our synchronization conditions in Section \ref{sec: synchronization
  conditions}.

\begin{definition}{\bf \emph{(Characteristic matrix)}}\label{def:
    characteristic subspace}
  For the network of oscillators $\mc G = (\mc V, \mc E)$ and the
  partition $\mc P = \{\mc P_1, \dots , \mc P_m\}$, the characteristic
  matrix of $\mc P$ is $V_{\mc P} \in \real^{n \times m}$, where
  \begin{align*}
    V_{\mc P} =
    \begin{bmatrix}
      v_1 & v_2 & \cdots & v_{m}
    \end{bmatrix}
                           ,
  \end{align*}
  and
  \begin{align*}
    v_i^\transpose =
    \big[
    \underbrace{
    \begin{matrix}
      0 & 0 & \cdots & 0
    \end{matrix}}_{ \sum_{j=1}^{i-1} |\mc P_j| }
                       \;
    \underbrace{
    \begin{matrix}
      1 & 1 & \cdots & 1
    \end{matrix}}_{ |\mc P_{i}| }
                       \;
    \underbrace{
    \begin{matrix}
      0 & 0 & \cdots & 0
    \end{matrix}}_{\sum_{j = i+1}^n |\mc P_{j}| }
    \big]
                       .
  \end{align*}
  \oprocend
\end{definition}

We conclude this section with an illustrative example.

\begin{figure}[t]
  \centering
  \includegraphics[width=.9\columnwidth]{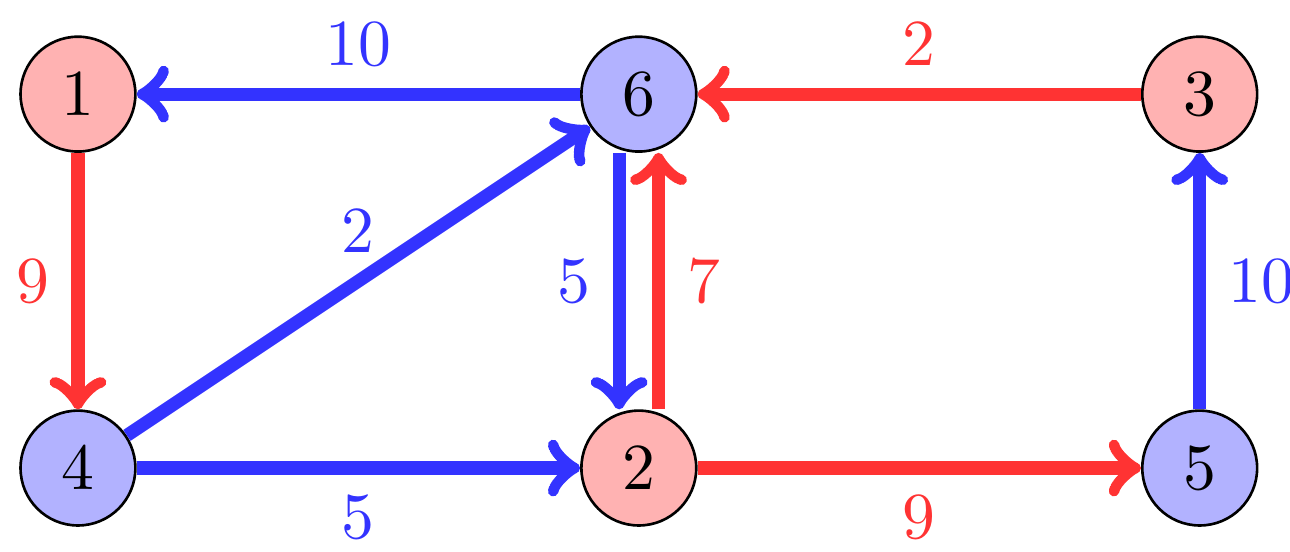}
  \caption{A network of oscillators with partitions
    $\mc P_1 = \{1,2,3\}$ and $\mc P_2 = \{4,5,6\}$. The sum of the
    weights of all edges $(i,j)$ is equal for each node $i$ of
    $\mc P_1$ (resp.  $\mc P_2$), with $j \in \mc P_2$ (resp.
    $j \in \mc P_1$). In Section \ref{sec: synchronization conditions}
    we show that this is a necessary condition for phase
    synchronization of the partition~$\mc P$.}
  \label{fig: example cluster synchronization}
\end{figure}

\begin{example}{\bf \emph{(Setup and definitions)}}\label{example:
    setup and definition}
  Consider the network of Kuramoto oscillators in Fig. \ref{fig:
    example cluster synchronization}, with graph
  $\mc G = (\mc V, \mc E)$, $\mc V = \{1,2,3,4,5,6\}$ and 
  partition $\mc P = \{\mc P_{1}, \mc P_{2} \}$. The graph $\mc G$ and
  the partition $\mc P$ are described by $A$ and $V_{\mc P}$ as
  follows:
  \begin{align*}
    A = 
    \begin{bmatrix}
      0 & 0 & 0 & 0 & 0 & 10\\
      0 & 0 & 0 & 5 & 0 & 5\\
      0 & 0 & 0 & 0 & 10 & 0\\
      9 & 0 & 0 & 0 & 0 & 0\\
      0 & 9 & 0 & 0 & 0 & 0\\
      0 & 7 & 2 & 2 & 0 & 0
    \end{bmatrix}
                          ,
                          V_\mc P = 
                          \begin{bmatrix}
                            1 & 0\\
                            1 & 0\\
                            1 & 0\\
                            0 & 1\\
                            0 & 1\\
                            0 & 1
                          \end{bmatrix}.
  \end{align*}
  \oprocend
\end{example}

\begin{figure}[tb]
  \centering \subfigure[]{
    \includegraphics[width=.9\columnwidth]{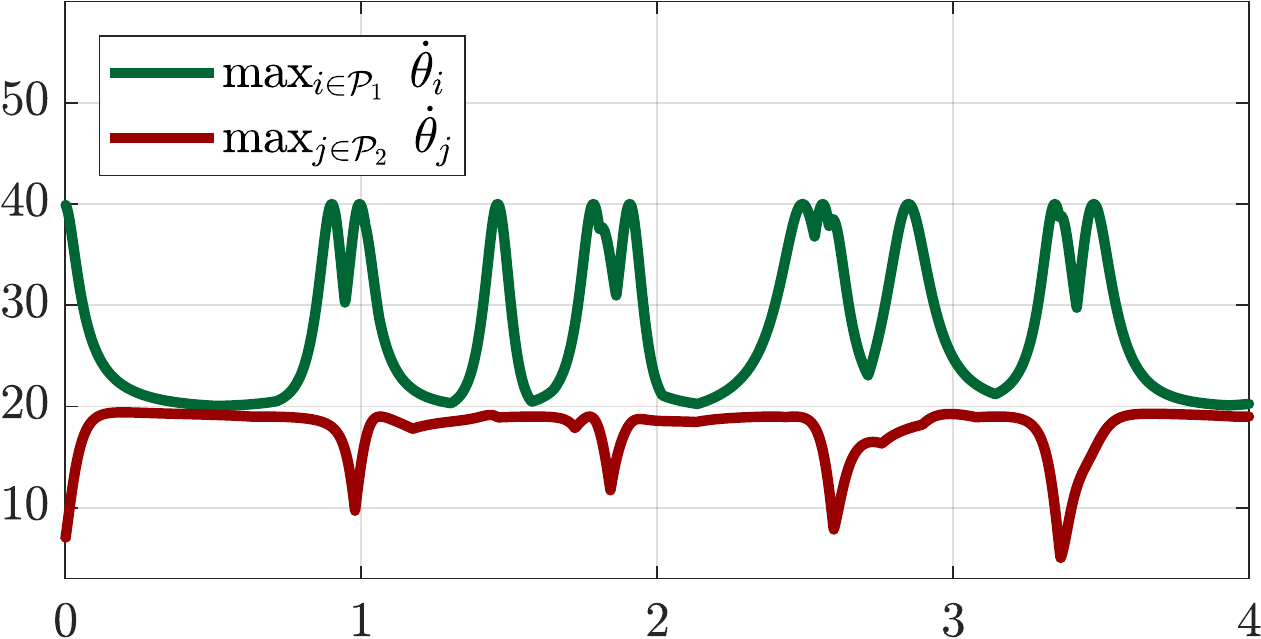}
    \label{fig: frequency no sync}
  } \subfigure[]{
    \includegraphics[width=.9\columnwidth]{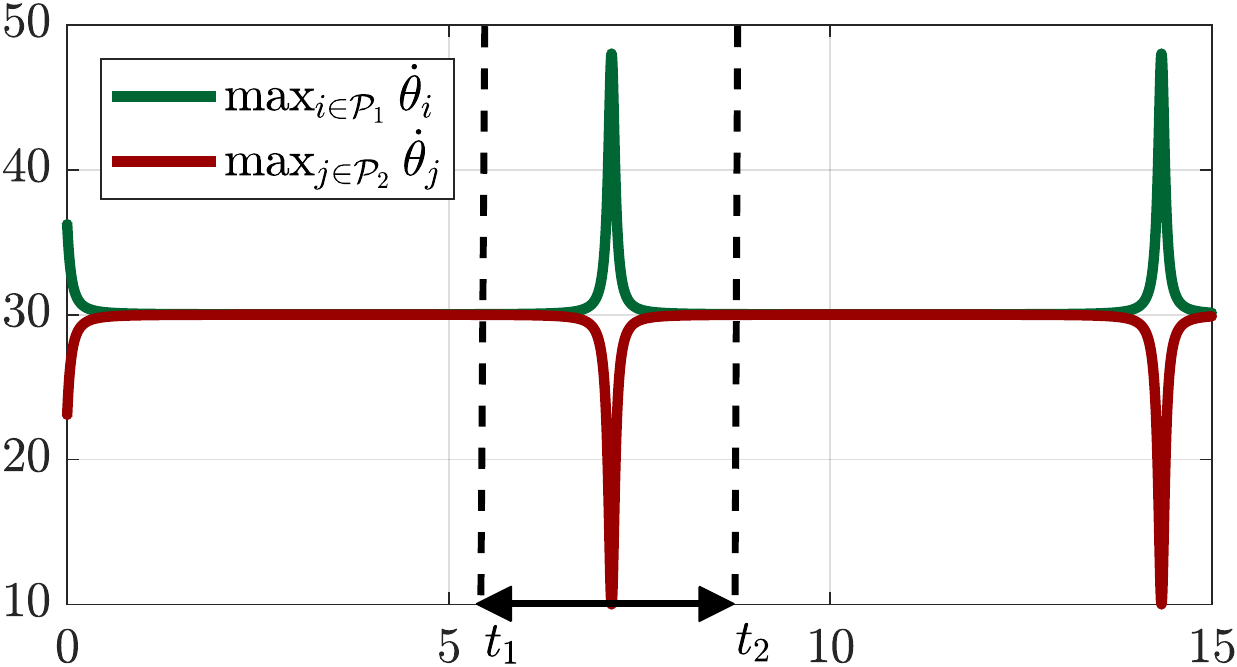}
    \label{fig: frequency sync}
  }
  \caption[Optional caption for list of figures]{For the network in
    Example \ref{fig: example cluster synchronization} with natural
    frequencies $\omega = [30,30,30,10,10,10]^\transpose$ Fig.  (a)
    shows the frequencies of the oscillators in the clusters
    $\mc P_1 = \{1,2,3\}$ and $\mc P_2 = \{4,5,6\}$ as a function of
    time. Notice that Assumption A1 is satisfied over the entire time
    interval. In Fig. (b) we let the frequencies be more homogeneous,
    $\omega = [19,19,19,10,10,10]^\transpose$. Assumption A1 is
    satisfied within bounded time intervals, such as $[t_1, t_2]$.}
  \label{fig:assumption A1}
\end{figure}

\section{Conditions for cluster synchronization}\label{sec:
  synchronization conditions} In this section we derive necessary and
sufficient conditions ensuring phase (hence frequency) synchronization
of a partition of oscillators. In particular, we show how
synchronization of a partition depends both on the interconnection
structure and weights, as well as the oscillators natural
frequencies. We make the following technical assumption.

\begin{itemize}
\item[(A1)] For the partition $\mc P = \{ \mc P_1, \dots, \mc P_m \}$
  there exists an ordering of the clusters $\mc P_i$ and an interval
  of time $[t_1, t_2]$, with $t_2 > t_1$, such that for all times
  $t \in [t_1, t_2]$:
  \begin{align*}
    \max_{i \in \mc P_1} \dot \theta_i > \max_{i \in \mc P_2} \dot
    \theta_i > \cdots > \max_{i \in \mc P_m} \dot \theta_i .
  \end{align*}

\end{itemize}

Assumption (A1) requires the phases of the oscillators in different
clusters to evolve with different frequencies, at least in some
interval of time. This assumption is in fact not restrictive, as this
is typically the case when the oscillators in different clusters have
different natural frequencies. Two cases where this assumption is
satisfied are presented in Fig.~\ref{fig:assumption A1}. A special
case where Assumption (A1) is not satisfied is discussed at the end of
this section.


\begin{theorem}{\bf \emph{(Cluster
      synchronization)}}\label{thm:cluster sync}
  For the network of oscillators $\mc G = (\mc V, \mc E)$, the
  partition $\mc P = \{ \mc P_1, \dots , \mc P_m \}$ is phase
  synchronizable if and only if the following conditions are
  simultaneously satisfied:

  \begin{enumerate}\setlength\itemsep{.5em}
  \item the network weights satisfy
    $\sum_{k \in \mc P_\ell} a_{i k} - a_{j k} = 0$ for every
    $i,j \in \mc P_z$ and $z,\ell \in \until{m}$, with $z \neq \ell$;

  \item the natural frequencies satisfy $\omega_i = \omega_j$ for
    every $k~\in~\until{m}$ and $i,j \in \mc P_k$.
  \end{enumerate}
\end{theorem}
\begin{proof}
  \textit{(If)} Let $\theta_i = \theta_j $ for all
  $i,j \in \mc P_k$, $k = 1, \dots, m$. Let $i,j \in \mc P_\ell$, and
  notice that
  \begin{align*}
    \dot \theta_i - \dot \theta_j &= 
    \sum_{z\neq \ell} \sum_{k
    \in \mc P_z} a_{ik}
    \sin(\theta_k -  \theta_i) -
    a_{jk} \sin(\theta_k - \theta_j)\\
    &= \sum_{z\neq \ell} s_{z\ell} \sum_{k \in \mc P_{z}} a_{ik} -
      a_{jk} = 0,
  \end{align*}
  where we have used conditions (i) and (ii), and where
  $s_{zl} = \sin(\theta_z - \theta_\ell)$ depends on the clusters $z$
  and $\ell$ but not on $i,j,k$. Thus, when conditions (i) and (ii)
  are satisfied, $\theta \in \Image(V_{\mc P})$ implies
  $\dot \theta \in \Image(V_{\mc P})$, the image of $V_{\mc P}$.
  $\Image (V_{\mc P})$ is invariant and the network is phase
  synchronizable ($\theta (0) \in \Image (V_{\mc P})$).

  \textit{(Only if)} We first show that condition (i) is necessary for
  phase synchronization. Assume that the network is phase
  synchronized. Let $i,j \in \mc P_\ell$. At all times is must hold
  that
  \begin{align}\label{eq: nec freq}
    \begin{split}
      0 = \ddot \theta_i - \ddot \theta_j =& \sum_{z \neq \ell} \sum_{k
        \in \mc P_z} a_{ik} \cos(\theta_k - \theta_i) (\dot \theta_k -
      \dot \theta_i)\\
      &-\sum_{z \neq \ell} \sum_{k
        \in \mc P_z} a_{jk} \cos(\theta_k - \theta_j) (\dot \theta_k -
      \dot \theta_j)\\
      &= \sum_{z \neq \ell} c_{z \ell} v_{z \ell} \underbrace{\sum_{k \in \mc P_z}
        a_{ik} - a_{jk}}_{d_z} ,
    \end{split}
  \end{align}
  where $c_{z \ell} = \cos(\theta_z - \theta_\ell)$ and
  $v_{z \ell} = \dot{\theta}_z - \dot{\theta}_{\ell}$ depend on the
  clusters $z$ and $\ell$, but not on $i,j,k$. From (A1), possibly
  after reordering the clusters, in some nontrivial interval we have
  \begin{align*}
    \max_{i \in \mc P_1} \dot \theta_i > \max_{i \in \mc P_2} \dot
    \theta_i > \cdots > \max_{i \in \mc P_m} \dot \theta_i .
  \end{align*}
  Thus, \eqref{eq: nec freq} implies that, either $d_z = 0$ for all
  $z$ (thus implying condition (i)), or the functions $c_{z \ell}v_{z \ell}$
  must be linearly dependent at all times in the interval. Assume by
  contradiction that the functions $c_{zl}v_{zl}$ are linearly
  dependent at all times in the above interval. Then it must hold that
  \begin{align*}
    \sum_{z \neq \ell} d_z \frac{d^n}{dt^n} c_{z \ell} v_{z \ell} = 0,
  \end{align*}
  for every nonnegative integer $n$, where $\frac{d^n}{dt^n}$ denotes
  $n$-times differentiation. In other words, not only the functions
  $c_{z \ell}v_{z \ell}$ must be linearly dependent, but also all their
  derivatives, at some times in the above interval. Let $d_1 \neq 0$
  (if $d_1 = 0$, simply select the first nonzero coefficient), and
  $i,j \not\in \mc P_1$. Because of assumption (A1), there exists an
  integer $n$ such that
  $d_1 \frac{d^n}{dt^n} c_{1 \ell} v_{1 \ell} \gg
  d_{z}\frac{d^n}{dt^n} c_{z \ell} v_{z \ell}$, for all $z \neq
  1$. Thus, the functions $c_{z \ell}v_{z \ell}$ cannot be linearly
  dependent. We conclude that statement (i) is necessary for phase
  synchronization.

  We now prove that, when the network is phase synchronized, statement
  (i) implies statement (ii). This shows that statement (ii) is
  necessary for phase synchronization. Let the network be phase
  synchronized, and let $i,j \in \mc P_\ell$. We have
  \begin{align*}
    0 = \dot \theta_i - \dot \theta_j = \omega_i - \omega_j + \sum_{z
    \neq \ell} s_{z \ell} \underbrace{\sum_{k \in \mc P_z} a_{ik} -
    a_{jk}}_{=0} ,
  \end{align*}
  where $s_{z \ell} = \sin(\theta_z - \theta_{\ell})$ does not depend on the indices $i, j, k$ (see
  above), and where we have used that statement (i) is necessary for
  phase synchronization. To conclude, $\omega_i = \omega_j$, and
  statement (ii) is also necessary for phase synchronization.
\end{proof}

\begin{remark}{\bf \emph{(Necessity of assumption
      A1)}}\label{remark: assumption}
  Consider a network of oscillators with adjacency matrix
  \begin{align*}
    A =
    \begin{bmatrix}
      0 & a_{12} & 0 & 0\\
      a_{21} & 0 & a_{23} & 0\\
      0 & a_{32} & 0 & a_{34}\\
      0 & 0 & a_{43} & 0\\ 
    \end{bmatrix}
    ,
  \end{align*}
  and natural frequencies $\omega_i = \bar \omega$ for all
  $i\in \until{4}$. Notice that condition (i) in Theorem
  \ref{thm:cluster sync} is not satisfied. Let
  $\theta_1 (0)= \theta_2 (0)$ and
  $\theta_3 (0)= \theta_4 (0)= \theta_1 (0)+ \pi$, and notice that
  $\dot \theta_i = \bar \omega$ at all times and for all
  $i\in \until{4}$ (Assumption A1 is not satisfied). In other words,
  the partition $\mc P = \{\mc P_1,\mc P_2\}$, with
  $\mc P_1 = \{1,2\}$ and $\mc P_2 = \{3,4\}$ is phase synchronized,
  independently of the interconnection weights among the
  oscillators. Thus, condition (i) in Theorem may not be necessary
  when Assumption A1 is not satisfied.  \oprocend
\end{remark}

Let $A \odot B$ denote the Hadamard product between $A$ and $B$
\cite{RAH-CRJ:85}, and $\Image(V_{\mc P})^\perp$ the orthogonal
subspace to $\Image(V_{\mc P})$.

\begin{corollary}{\bf \emph{(Matrix condition for
      synchronization)}}\label{crl:condition equivalence}
  Condition $(i)$ in Theorem \ref{thm:cluster sync} is equivalent to
  $\bar{V}_{\mc P}^\transpose \bar{A} V_{\mc P} = 0$, where $\bar V_{\mc P}\in\mathbb{R}^{n\times(n-m)}$ satisfies
  $\Image(\bar V_{\mc P}) = \Image(V_{\mc P})^\perp$, and
  \begin{align}\label{eq: bar A}
    \bar{A} = A-A\odot V_{\mc P}V_{\mc P}^\transpose .
  \end{align}
\end{corollary}
\vspace{2mm}
\begin{proof}
  Let $\bar A = [\bar a_{ij}]$ and $A = [a_{ij}]$. Notice that
  $\bar a_{ij} = a_{ij}$ when $i$ and $j$ belong to different
  clusters, and $\bar a_{ij} = 0$ when $i$ and $j$ belong to the same
  cluster. Thus,
  \begin{align*}
    [\bar{A}V_{\mc P}]_{ij} = \begin{cases}
      \sum_{k \in \mc P_j} a_{ik} , & \text{if } i \notin \mc P_{j},\\
      0, & \text{if } i \in \mc P_{j} .
    \end{cases}
  \end{align*}
  Select $\bar{V}_{\mc P}$ so that
  $\bar{V}_{\mc P} = [\bar v_1 \cdots \bar v_{n-m} ]$ and
  $\bar v_i^\transpose x = x_r - x_s$, with $r,s \in \mc P_\ell$,
  for a vector $x$ of compatible dimension. Then,
  \begin{align*} 
    [\bar{V}_{\mc P}^\transpose\bar{A}V_{\mc P}]_{ij}=
    \begin{cases}
      \sum_{k \in \mc P_{j}} a_{rk}-a_{sk} , & r,s \notin \mc P_j ,\\
      0, & r,s \in \mc P_j ,
    \end{cases}
  \end{align*}    
  where $r,s$ are the nonzero indices of $\bar v_{i}$.
\end{proof}

\section{Control of cluster synchronization}\label{sec: control of
  cluster synchronization}
In the previous section we derive conditions on the network of
oscillators to guarantee phase and frequency synchronization. These
conditions are rather stringent, and are typically not satisfied for
arbitrary partitions and interconnection weights. In this section we
develop a control mechanism to modify the oscillators' interconnection
weights so as to guarantee synchronization of a given
partition. Specifically, we study the following minimization problem:
\begin{subequations}
\label{eq:optimproblem}
\begin{align}
  \min_{\Delta} \;\;\;\;\;\;&\|\Delta\|^{2}_{F}\\
  \text{s.t.} \;\;\;\;\;\;&\bar{V}_{\mc P}^{\transpose} \left[\bar{A} + \Delta\right]V_{\mc P} = 0 \label{constraint1}\\
                            &\Delta \in \mc{H} \label{constraint2}
\end{align}
\end{subequations}
where $\|\Delta\|_{F}$ denotes the Frobenius norm of the matrix
$\Delta$, $\bar A$ is as in \eqref{eq: bar A}, and $\mc H$ encodes a
desired sparsity pattern of the perturbation matrix $\Delta$. For
example, $\mc H$ may represent the set of matrices compatible with the
graph $\mc G = (\mc V, \mc E)$, that is,
$\mc H = \setdef{M}{M \in \real^{|\mc V| \times |\mc V|} \text{ and
  }m_{ij}= 0 \text{ if } (i,j)\not \in \mc E}$. The constraint
\eqref{constraint1} reflects the invariance condition in Corollary
\ref{crl:condition equivalence} and, together with condition (ii) in
Theorem \ref{thm:cluster sync}, ensures synchronization of the
partition $\mc P$. Thus, the minimization problem
\eqref{eq:optimproblem} determines the smallest perturbation of the
interconnection weights that guarantees synchronization of a partition
$\mc P$ and satisfies desired sparsity constraints. It should be
observed that, given the solution $\Delta^*$ to
\eqref{eq:optimproblem}, the modified adjacency matrix is $A+\Delta^*$
even if the constraint \eqref{constraint1} is expressed in terms of
$\bar{A}$. This follows from the fact that connections among nodes of
the same cluster do not affect the synchronization properties of the
partition $\mc P = \{\mc P_1, \dots , \mc P_m\}$ (see Corollary
\ref{crl:condition equivalence}). 

To solve the minimization problem \eqref{eq:optimproblem}, we define
the following minimization problem by including the sparsity
constraints \eqref{constraint2} into the cost function:
\begin{subequations}
  \label{eq:optimproblem1} \begin{align}
  \min_{\Delta} \;\;\;\;\;\;&\|\Delta \oslash H\|^{2}_{F}\\
  \text{s.t.} \;\;\;\;\;\;&\bar{V}_{\mc P}^{\transpose} \left[\bar{A} +
                          \Delta\right]V_{\mc P} = 0  
\end{align}
\end{subequations}
where $\oslash$ denotes elementwise division, and $H$ satisfies
$h_{ij} = 1$ if there exists a matrix $M \in \mc H$ such that
$m_{ij} \neq 0$, and $h_{ij} = 0$ otherwise. Clearly, the minimization
problems \eqref{eq:optimproblem} and \eqref{eq:optimproblem1} are
equivalent, in the sense that $\Delta^*$ is a (feasible) solution to
\eqref{eq:optimproblem} if and only if it has finite cost in
\eqref{eq:optimproblem1}.

\begin{theorem}{\bf \emph{(Synchronization via structured
      perturbation)}}\label{thm:optim constraints}
  Let $T = [V_{\mc P} \; \bar V_{\mc P}]$, and let 
  \begin{align*}
    \begin{bmatrix}
          \tilde A_{11} & \tilde A_{12}\\
          \tilde A_{21} & \tilde A_{22}
        \end{bmatrix} = T^{-1} \bar A T.               
  \end{align*}
  The minimization problem \eqref{eq:optimproblem} has a solution if
  and only if there exists a matrix $\Lambda$ satisfying 
  \begin{align*}
    X &= (\bar V_\mc P \Lambda V_\mc P^{\transpose})\odot H, \text{
        and }
    \tilde A_{21} = \bar V_\mc P^{\transpose} X V_\mc P .
  \end{align*}

  Moreover, if it exists, a solution $\Delta^*$ to
  \eqref{eq:optimproblem} is
    \begin{align*}
    \Delta^* = T
    \begin{bmatrix}
      \tilde \Delta_{11}^* &  \tilde \Delta_{12}^*\\
      \tilde \Delta_{21}^* & \tilde \Delta_{22}^*
    \end{bmatrix}
                            T^{-1},
  \end{align*}
  where $\tilde \Delta_{11}^* = -V_\mc P^{\transpose}XV_\mc P$,
  $\tilde \Delta_{12}^* = -V_\mc P^{\transpose}X\bar V_\mc P$,
  $\tilde \Delta_{21}^* = -\tilde A_{21}$, and
  $\tilde \Delta_{22}^* = -\bar V_\mc P^{\transpose}X\bar V_\mc P$.
\end{theorem}

\begin{proof}
  We adopt the method of Lagrange multipliers to derive the optimality
  conditions for the problem \eqref{eq:optimproblem1}. The Lagrangian
  is
\begin{align*}
  \mc{L}(\Delta, \Lambda) = & \sum_{i = 1}^{n} \sum_{j=1}^{n} \delta_{ij}^2 h_{ij}^{-1} + \sum_{i=1}^{m} \lambda_i^{\transpose}\bar V_\mc P^{\transpose}(\bar{A}+\Delta) v_{i},
\end{align*}
where
$\Lambda= [ \lambda_1,\dots,\lambda_m ] \in \mathbb{R}^{(n-m)\times
  m}$ is a matrix collecting vectors of Lagrange multipliers, and
$v_i \in \mathbb{R}^{n}$ is the $i$-th column of $V_{\mc P}$. By equating
the partial derivatives of $\mc L$ to zero we obtain the following
optimality conditions:
\begin{subequations}
\begin{align}\label{eq:optimality condition1}
\frac{\partial \mc L}{\partial \lambda_i}&= 0  \Rightarrow  \bar V_\mc P (\bar{A}+\Delta)v_{i}=0,\\
\label{eq:optimality condition2}
\frac{\partial \mc L}{\partial \delta_{ij}}&= 0   \Rightarrow  2\delta_{ij}h_{ij}^{-1} +\sum_{k=1}^{m} \lambda_k^{\transpose} \bar v_i^\transpose v_{jk} = 0,
\end{align}
\end{subequations}
where $\bar v_i$ is the $i$-th row of $\bar{V}_{\mc P}$ and $v_{jk}$
is the entry $(j,k)$ of the matrix $V_\mc P$.  Finally,
\eqref{eq:optimality condition1} and \eqref{eq:optimality condition2}
can be rewritten as
\begin{subequations}\label{eq: opt}
\begin{align}\label{eq:optimality cond matrix 1}
\bar V_\mc P^\transpose (\bar{A}+\Delta)V_\mc P & = 0,\\
\Delta\oslash H+ \bar V_\mc P \Lambda V_\mc P^{\transpose} & = 0, \label{eq:optimality cond matrix 2}
\end{align}
\end{subequations}
where the factor 2 of \eqref{eq:optimality condition2} has been
included into the Lagrange multipliers. Applying the change of
coordinates $T = [V_{\mc P} \; \bar V_{\mc P}]$,
$\bar{A} = T \tilde{A} T^{-1}$ and $\Delta = T\tilde{\Delta}T^{-1}$,
with $I_d$ the identity matrix of dimension $d$, equation
\eqref{eq:optimality cond matrix 1} becomes
\begin{align*}
& \bar V_\mc P^\transpose T(\tilde A+\tilde\Delta)  T^{-1} V_\mc P=\\
& \begin{bmatrix}
0 & I_{n-m}\end{bmatrix} \begin{bmatrix}
\tilde A_{11}+\tilde{\Delta}_{11} & \tilde A_{12}+\tilde{\Delta}_{12}\\
\tilde A_{21}+\tilde{\Delta}_{21} & \tilde A_{22}+\tilde{\Delta}_{22}
\end{bmatrix}
\begin{bmatrix}
I_{m}\\ 0
\end{bmatrix}= 0,
\end{align*}
which leads to $\tilde \Delta_{21}^{*} = -\tilde A_{21}$.
Equation \eqref{eq:optimality cond matrix 2} is equivalent to
$\Delta+ (\bar V_\mc P \Lambda V_\mc P^{\transpose})\odot H = 0$,
which can be decomposed as
\begin{align*}
\underbrace{\begin{bmatrix}V_\mc P & \bar V_\mc P\end{bmatrix}}_{T}
\begin{bmatrix}
\tilde \Delta_{11} & \tilde \Delta_{12}\\
\tilde \Delta_{21} & \tilde \Delta_{22}
\end{bmatrix}
\underbrace{\begin{bmatrix}
V_\mc P^\transpose\\
\bar V_\mc P^\transpose
\end{bmatrix}}_{T^{-1}} + (\bar V_\mc P \Lambda V_\mc P^{\transpose})\odot H    = 0.
\end{align*}
Consequently,
\begin{equation}\label{eq: decomposition}
\begin{split}
(V_\mc P\tilde \Delta_{11}V_\mc P^\transpose & - \bar V_\mc P \tilde A_{12} V_\mc P^\transpose + V_\mc P \tilde \Delta_{12}\bar V_\mc P^\transpose + \bar V_\mc P \tilde \Delta_{22}\bar V_\mc P^\transpose) +\\
&  (\bar V_\mc P \Lambda V_\mc P^{\transpose})\odot H    = 0.
\end{split}
\end{equation}

Let $X =(\bar V_\mc P \Lambda V_\mc P^{\transpose})\odot H$. Recall
that $\bar V_\mc P^\transpose \bar V_\mc P= I_{n-m}$,
$V_\mc P^\transpose V_\mc P=I_m$, and
$\bar V_\mc P^\transpose V_\mc P= 0$. By pre-multiplicating equation
\eqref{eq: decomposition} by $\bar V^{\transpose}_\mc P$ and
post-multiplicating it by $V_\mc P$, we obtain
\begin{subequations}
\begin{align*}
-\tilde A_{21} + \bar V_\mc P^\transpose X V_\mc P = 0,
\end{align*}
which is a system of linear equations, that can be solved with respect
to the unknown $\Lambda$. Following the same reasoning of above, we
can obtain the following other three equations that entirely determine
the solution $\tilde \Delta_{11}$, $\tilde \Delta_{12}$, and
$\tilde \Delta_{22}$:
\begin{align*}
  \tilde \Delta_{11} + V_\mc P^{\transpose}XV_\mc P & = 0,\;
                                                      \tilde \Delta_{12} + V_\mc P^{\transpose}X\bar V_\mc P  = 0, \text{
                                                      and }\\ 
  \tilde \Delta_{22} + \bar V_\mc P^{\transpose}X\bar V_\mc P  &= 0.
\end{align*}
\end{subequations}
Finally, the optimal matrix $\Delta^{*}$, solution to the problem
\eqref{eq:optimproblem1}, is given in original coordinates as
\begin{align*}
\Delta^{*} = T\begin{bmatrix}
\tilde \Delta^{*}_{11} & \tilde \Delta^{*}_{12}\\
-\tilde A_{21} & \tilde \Delta^{*}_{22}
\end{bmatrix}T^{-1}.
\end{align*}

\end{proof}

\begin{figure*}[tb]
  \centering \subfigure[]{
    \includegraphics[width=.64\columnwidth]{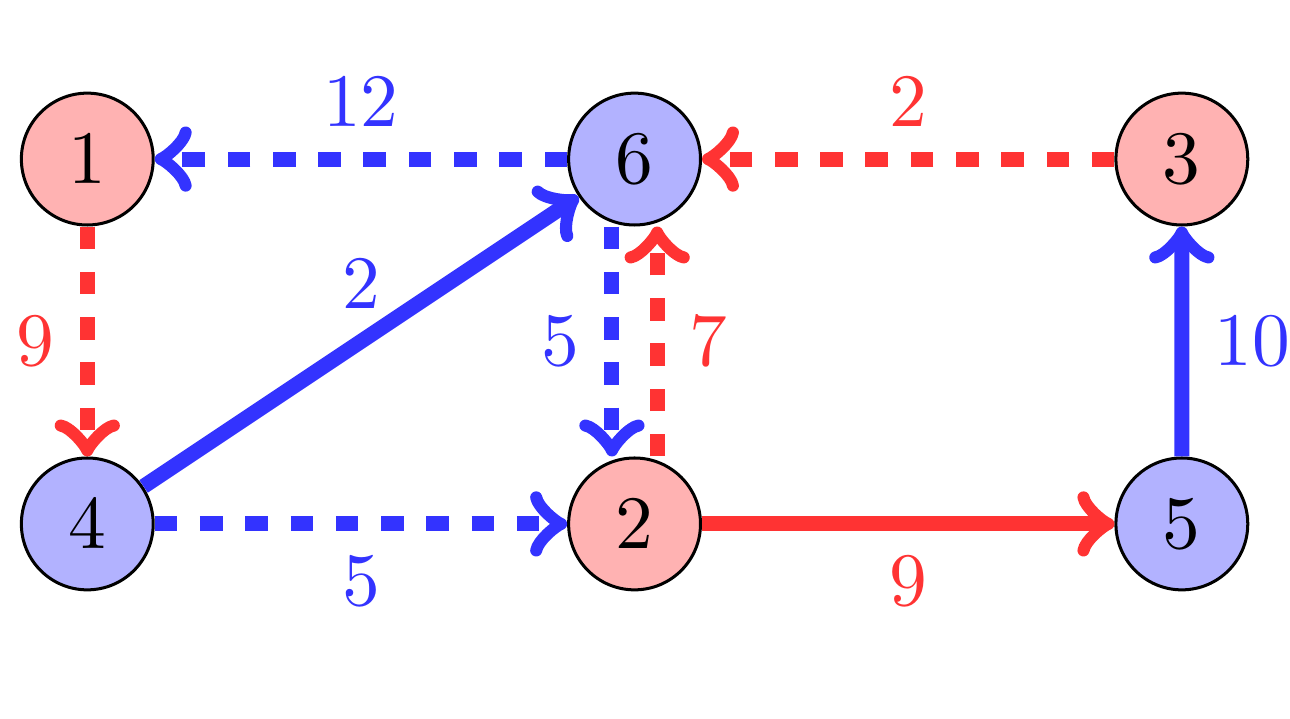}
    \label{fig: frequency no sync}
  } \subfigure[]{
    \includegraphics[width=.64\columnwidth]{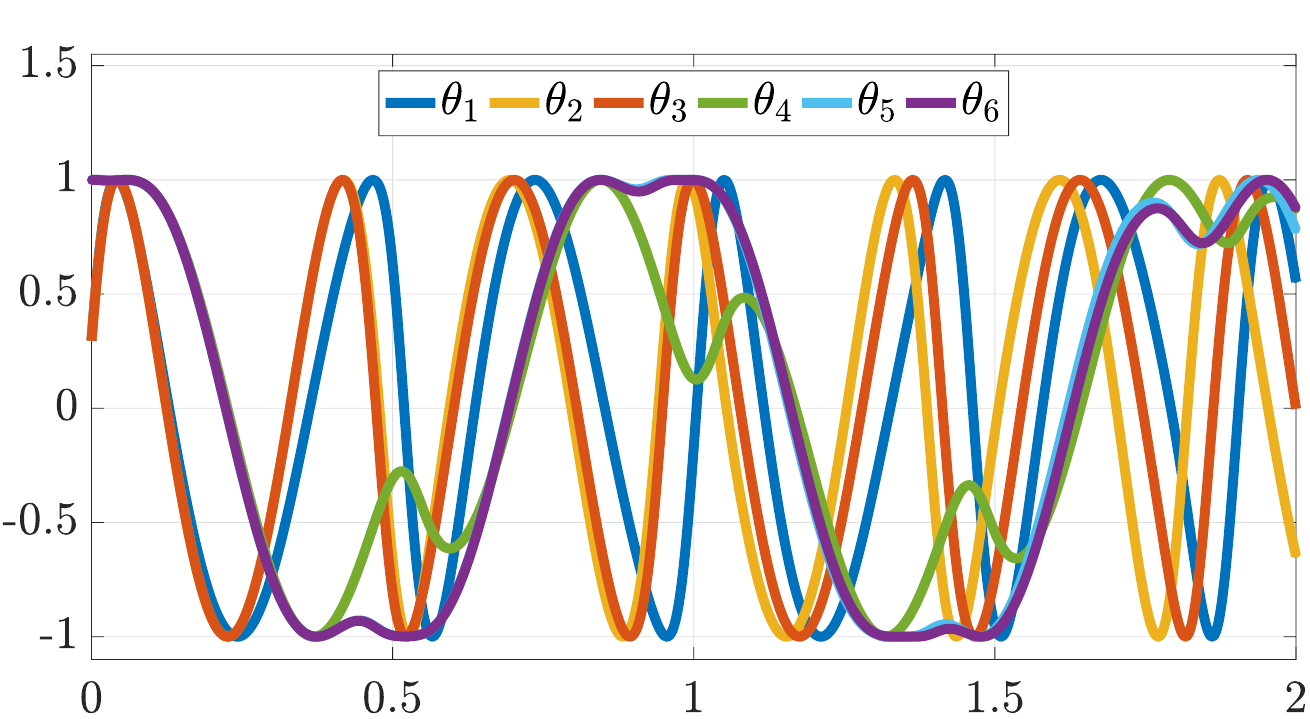}
    \label{fig: frequency sync}
  }
\subfigure[]{
    \includegraphics[width=.64\columnwidth]{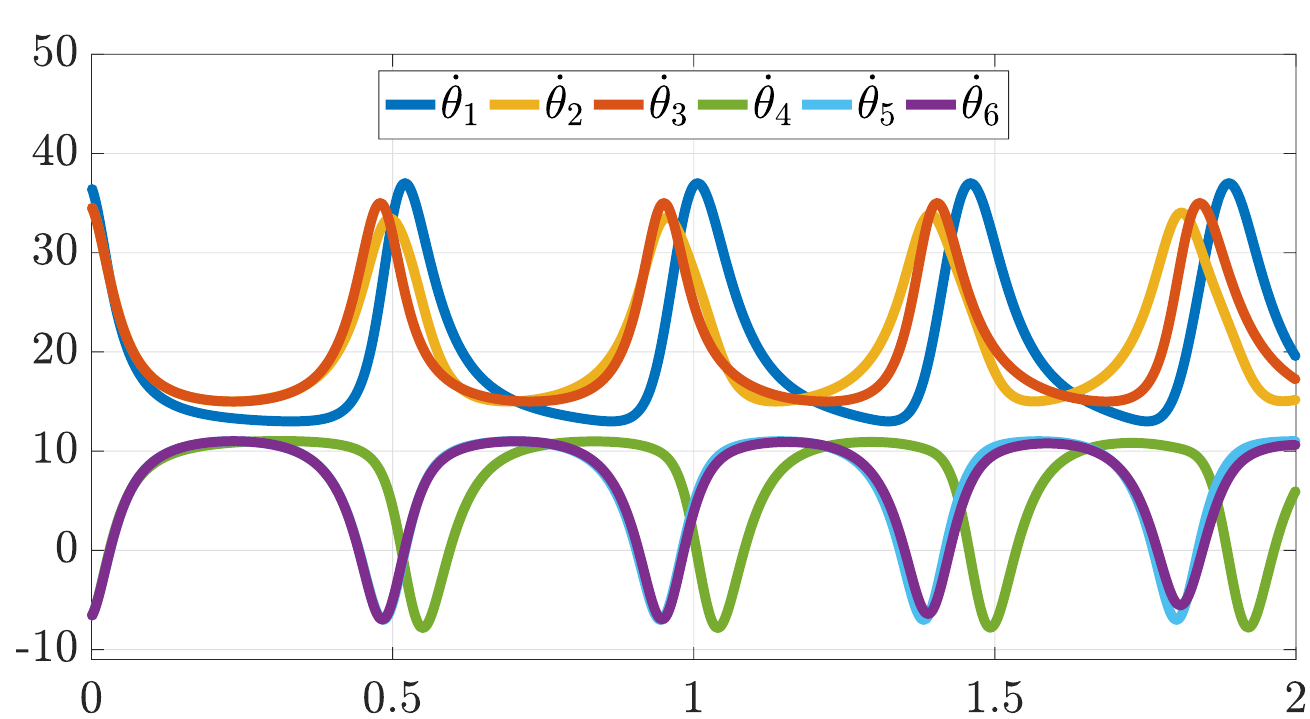}
    \label{fig: frequency sync}
  }
  \\
  \subfigure[]{
    \includegraphics[width=.64\columnwidth]{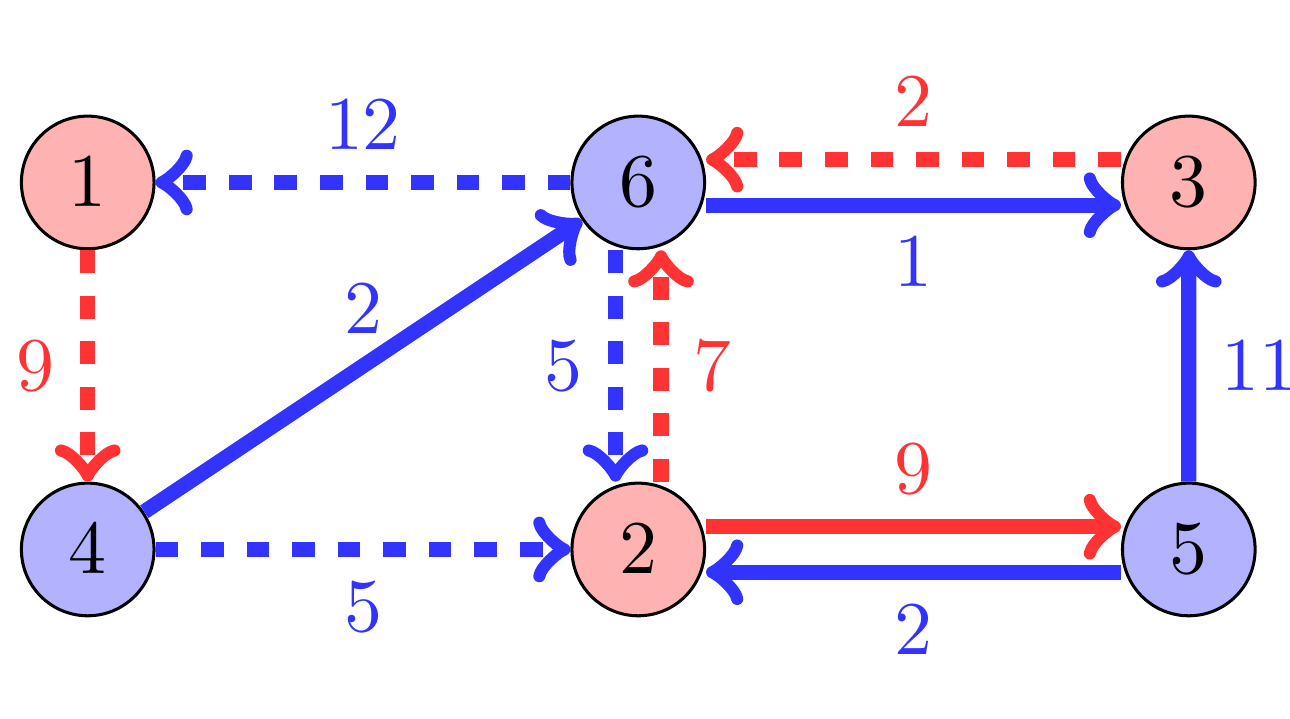}
    \label{fig: frequency no sync}
  } \subfigure[]{
    \includegraphics[width=.64\columnwidth]{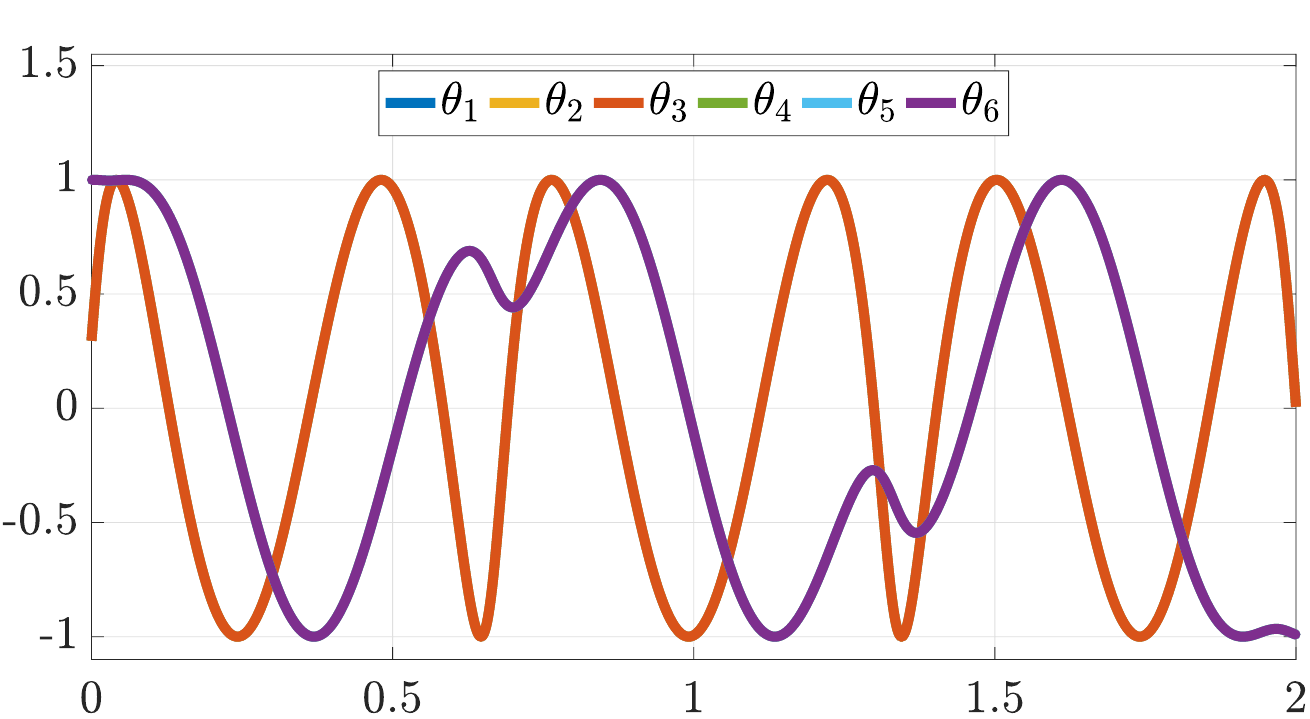}
    \label{fig: frequency sync}
  }
  \subfigure[]{
    \includegraphics[width=.64\columnwidth]{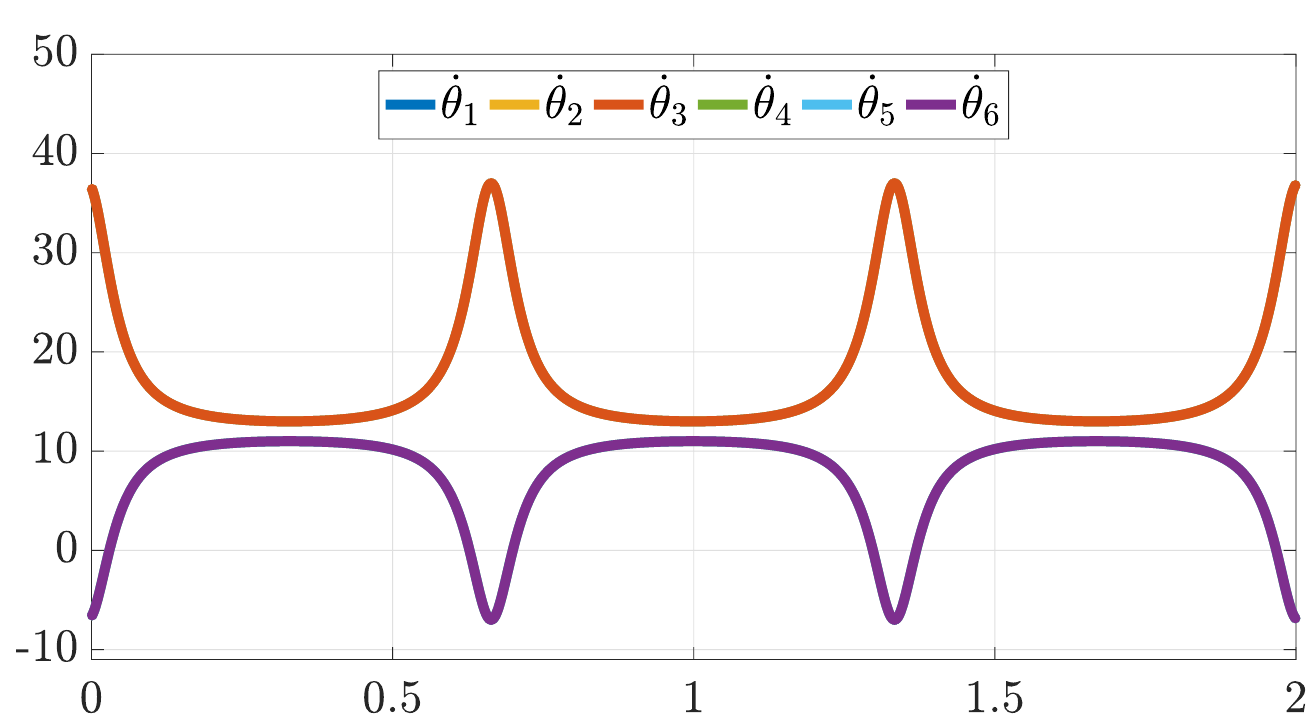}
    \label{fig: frequency sync}
  }
  \caption[Optional caption for list of figures]{Fig. (a) shows the
    network in Example \ref{example: sync network}, where the dashed
    (resp. solid) edges correspond to the zero (resp. unit) entries of
    $H$. The partition $\mc P = \{\mc P_1, \mc P_2\}$, with
    $\mc P_1 = \{1,2,3\}$ and $\mc P_2 = \{4,5,6\}$, is not
    synchronizable because, for instance, the sum of the weights of
    the incoming edges to nodes $1$ and $2$ is different (see Theorem
    \ref{thm:cluster sync}). Fig. (b) and (c) show the phases and
    frequencies of the oscillators as a function of time. Fig. (d)
    shows the modified network obtained from Theorem \ref{thm:optim
      constraints}, which satisfies condition (i) in Theorem
    \ref{thm:cluster sync} and leads to a synchronizable partition
    $\mc P$. When the natural frequencies are selected to satisfy
    condition (ii) in Theorem \ref{thm:cluster sync}, the oscillators'
    phases and frequencies are synchronized as illustrated in Fig. (e)
    and (f).}
  \label{fig:example synchronization 2}
\end{figure*}

Theorem \ref{thm:optim constraints} characterizes the smallest
(measured by the Frobenius norm) structured network perturbation that
ensures synchronization of a given partition. Without constraints, the
optimal perturbation has a straightforward expression.

\begin{corollary}{\bf \emph{(Unconstrained minimization
      problem)}}\label{corollary:optim noconstr}
  Let $\mc H = \setdef{M}{m_{ij} \neq 0 \text{ for all } i \text{ and
    } j}$. The minimization problem \eqref{eq:optimproblem} is
  always feasible, and its solution is
  \begin{align*}
    \Delta^* = -\bar{V}_{\mc P} \bar{V}_{\mc P}^{\transpose}\bar AV_{\mc
    P}V_{\mc P}^{\transpose} .
    \end{align*}
\end{corollary}
\begin{proof}
  Because $h_{ij} = 1$ for all $i$ and $j$, the optimality condition
  \eqref{eq:optimality cond matrix 2} becomes
\begin{align*}
\Delta + \bar V_\mc P \Lambda V_\mc P^{\transpose} = 0.
\end{align*}
We now pre- and post-multiply both sides of the above equality by
$\bar V_\mc P^\transpose$ and $V_\mc P$, respectively, and obtain
\begin{subequations}
\begin{align*}
\Lambda & = \bar{V}_{\mc P}^{\transpose} \bar{A} V_{\mc P}, \;
\Delta^{*}  = -\bar{V}_{\mc P}\bar{V}_{\mc P}^{\transpose} \bar{A} V_{\mc P} V_{\mc P}^{\transpose},
\end{align*}
\end{subequations}
where we have used \eqref{eq:optimality cond matrix 1},
$V_\mc P^\transpose V_\mc P = I$, and
$\bar V_\mc P^\transpose \bar V_\mc P = I$.
\end{proof}

We now present an example where we modify the network weights to
ensure synchronization of a desired partition.

\begin{example}{\bf \emph{(Enforcing synchronization of a
      partition)}}\label{example: sync network}
  Consider the network in Fig. \ref{fig:example synchronization
    2}(a). The dashed edges and the solid edges represent constrained
  and uncostrained edges, respectively. The corresponding matrices
  $\bar{A}$ and $H$ read as {\small
  \begin{align*}
     \bar{A} = 
    \begin{bmatrix}
      0 & 0 & 0 & 0 & 0 & 12\\
      0 & 0 & 0 & 5 & 0 & 5\\
      0 & 0 & 0 & 0 & 10 & 0\\
      9 & 0 & 0 & 0 & 0 & 0\\
      0 & 9 & 0 & 0 & 0 & 0\\
      0 & 7 & 2 & 0 & 0 & 0
    \end{bmatrix}
                          ,
    H = 
    \begin{bmatrix}
      0 & 1 & 1 & 0 & 0 & 0\\
      1 & 0 & 1 & 0 & 1 & 0\\
      1 & 1 & 0 & 0 & 1 & 1\\
      0 & 1 & 1 & 0 & 1 & 1\\
      1 & 1 & 1 & 1 & 0 & 1\\
      1 & 0 & 0 & 1 & 1 & 0
    \end{bmatrix}                .
  \end{align*}}
  Notice that $H$ allows only a subset of interconnections
  to be modified, specifically, those corresponding to its unit entries.

  It can be shown that, because condition (i) in Theorem
  \ref{thm:cluster sync} is not satisfied (equivalently
  $\bar{V}_{\mc P}^{\transpose} \bar{A }V_{\mc P} \neq 0$), the
  network is not phase synchronizable (see Fig. \ref{fig:example
    synchronization 2}(b) and \ref{fig:example synchronization 2}(c) for
  an evolution of the oscillators' phases and frequencies). From
  Theorem \ref{thm:optim constraints} we obtain the optimal
  perturbation that ensures synchronization, which leads to the
  network in Fig. \ref{fig:example synchronization 2}(d). Notice that
  the network in Fig. \ref{fig:example synchronization 2}(d) satisfies
  condition (i) in Theorem \ref{thm:cluster sync}. In fact, when the
  natural frequencies are equal within each cluster (condition (ii) in
  Theorem \ref{thm:cluster sync}), the clusters evolve cohesively; see
  Fig. \ref{fig:example synchronization 2}(e) and \ref{fig:example
    synchronization 2}(f).
      \oprocend

\end{example}

\section{Conclusion}\label{future research and conclusions}
In this work we study cluster synchronization in networks of Kuramoto
oscillators. We derive necessary and sufficient conditions on the
network interconnection weights and on the oscillators' natural
frequencies to guarantee that the phases of groups of oscillators
evolve cohesively with one another, yet independently from the phases
of oscillators belonging to different groups. Additionally, we develop
a control mechanism to modify the edges of a network to ensure the
formation of desired clusters. Our control method is optimal, as it
determines the smallest perturbation (measured by the Frobenius norm)
for a desired synchronization pattern that is compatible with a
pre-specified set of structural constraints.


\bibliographystyle{unsrt}
\bibliography{./alias,./Main,./FP}

\end{document}